\documentclass{amsart}
\usepackage[margin=.8in]{geometry}

\usepackage{amssymb}
\usepackage{amsthm}
\usepackage{amsmath}
\usepackage{tikz}
\usepackage{float}
\usepackage{subcaption}

\newtheorem{thm}{Theorem}[section]

\newtheorem{lemma}[thm]{Lemma}
\newtheorem{cor}[thm]{Corollary}

\theoremstyle{definition}

\theoremstyle{remark}
\newtheorem{remark}[thm]{Remark}

\numberwithin{equation}{section}

\newcommand{\POS}{\mathcal{P}}
\newcommand{\SOS}{\mathcal{S}}
\DeclareMathOperator*{\tr}{tr}
\DeclareMathOperator*{\tw}{tw}
\DeclareMathOperator*{\Sym}{Sym}

\newcommand{\Z}{\mathbb{Z}}

\begin{document}

\title{Approximate PSD-Completion on Generalized Chordal Graphs}

\author{Kevin Shu}

\begin{abstract}
    Recently, there has been interest in the question of whether a partial matrix in which many of the fully defined principal submatrices are PSD is approximately PSD completable.
    These questions are related to graph theory because we can think of the entries of a symmetric matrix as corresponding to the edges of a graph.

    We first introduce a family of graphs, which we call thickened graphs; these contain both triangle-free and chordal graphs, and can be viewed as the result of replacing the edges of a graph by an arbitrary chordal graph.
    We believe these graphs might be of independent interest.

    We then show that for a class of graphs including thickened graphs, it is possible to get quantitative bounds on how well the property of having these principal submatrices being PSD approximates the PSD-completability property.
    These bounds frequently only depend on the size of the smallest cycle of size at least 4 in the graph.
    We introduce some tools that allow us to better control the quality of these approximations and indicate how these approximations can be used to improve the performance of semidefinite programs.
    The tools we use in this paper are an interesting mix of algebraic topology, structural graph theory, and spectral analysis.

\end{abstract}
\maketitle

\section{Introduction}
The space of PSD matrices is one of the basic objects that is studied in convex optimization because of its connection to semidefinite programming.
We will be interested in understanding coordinate projections of the set of PSD matrices, or equivalently, the theory of PSD matrix completion.

The entries of an $n\times n$ symmetric matrix $X$ are indexed by unordered pairs $\{i,j\}$ for $i,j \in [n]$. 
Given a graph $G$, there is a natural projection of $X$ onto just those entries corresponding to the edges of $G$.
We call the image of this projection the space of $G$-partial matrices; they are represented as matrices where entries corresponding to nonedges of $G$ are `forgotten'.
\begin{figure}[H]
\begin{subfigure}{0.5\textwidth}
  \centering
	\begin{tikzpicture}
        \node[circle,fill=black,minimum size=0.25mm](v1) at (0:1) {};
        \node[circle,fill=black,minimum size=0.25mm](v2) at (90:1) {};
        \node[circle,fill=black,minimum size=0.25mm](v3) at (180:1){};
        \node[circle,fill=black,minimum size=0.25mm](v4) at (270:1){};
        \node[circle,fill=black,minimum size=0.25mm](v5) at (360:1){};
        \draw (v1) -- (v2) -- (v3) -- (v4) -- (v5) -- cycle;
	\end{tikzpicture}
\end{subfigure}%
\begin{subfigure}{0.5\textwidth}
\[
  \begin{pmatrix}
    1 & 0 & 0 & 0\\
    0 & 1 & 0 & 0\\
    0 & 0 & 1 & 0\\
    0 & 0 & 0 & 1
  \end{pmatrix} \rightarrow
  \begin{pmatrix}
    1 & 0 & ? & 0\\
    0 & 1 & 0 & ?\\
    ? & 0 & 1 & 0\\
    0 & ? & 0 & 1
  \end{pmatrix}
\]
\end{subfigure}
\caption{An example of the projection of a matrix onto the edges of a cycle graph.}
\end{figure}
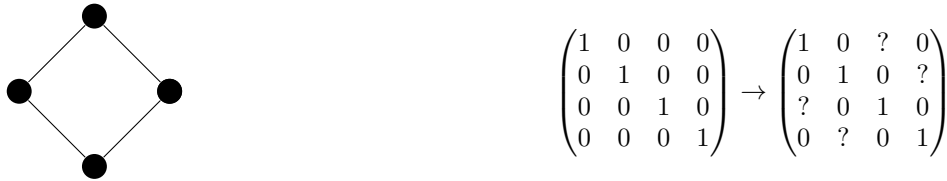

The image of the PSD matrices under this projection is precisely the set of $G$-partial matrices which can be completed to a PSD matrix. That is, these are partial matrices where we can choose the `forgotten' entries to make the resulting symmetric matrix PSD.
Checking that a partial matrix can be completed to a PSD matrix is known as the PSD completion problem, and it has been the subject of much study, for example in \cite{laurent2014positive}.

We can give a clear necessary condition for a matrix to be PSD completable: if we can find a subset $S$ of the vertices of $G$ where for all $i,j \in S$, $\{i,j\} \in G$, then we can form a submatrix of a partial matrix $X$ by just looking at the indices in $S$.
Being PSD is preserved under taking submatrices, so if $X$ is PSD completable, then this submatrix of $X$ must be PSD.
A subset $S$ with this property is a \textbf{clique} of $G$, and if for all cliques contained in $G$, the corresponding submatrix of $X$ is PSD, then we say that $X$ is \textbf{$G$-locally PSD}.

Checking that these smaller blocks are PSD is often easier than checking that the partial matrix is PSD completable.
We are interested in understanding whether it is sufficient to check that these smaller submatrices are PSD.

In the work of  Grone et al. in \cite{GRONE1984109}, it was shown that if $G$ is chordal, then $X$ is PSD completable if and only if $X$  is $G$-locally PSD, but that for all other graphs, there are $G$-locally PSD partial matrices which are not PSD completable.
There have been various results extending this equality result by introducing additional constraints, besides the submatrix-PSD constraints.
For example, Barrett and Laurent have considered the so-called cycle conditions for PSD completability in series parallel graphs, and given precise conditions for PSD completability in these cases in \cite{BARRETT19933} and \cite{MR1428642}.

More recently, we have been interested in approximate PSD-completability, and whether the $G$-locally PSD condition is enough to guarantee that a partial matrix is \textbf{approximately} PSD completable.
In \cite{blekherman2020sums}, these PSD matrix completion questions were connected to the theory of sum-of-squares and nonnegative quadratic forms on certain algebraic varieties.
There, quantitative results were shown on the distance between the $G$-locally PSD cone and the PSD-completable cone.
In \cite{shu2021extreme}, we strengthened the connection with the theory of nonnegative quadratic forms over algebraic varieties, which allowed us to uncover some strong structural properties of these $G$-locally PSD partial matrices.
This paper, which combines the types of questions asked in \cite{blekherman2020sums} with the structural results in \cite{shu2021extreme}, obtains approximation guarantees for a much more general class of graphs, which we hope will prove to be useful in practical applications.

In this paper, we will review the quantitative measurements of the distance between the $G$-locally PSD partial matrices and the PSD completable partial matrices in \cite{blekherman2020sums}.
We will define our class of thickened graphs in \ref{subsec:thick} and we will give our quantitative results in Theorem \ref{thm:eps_thick}.
These graphs are obtained by replacing the edges of a given graph by general chordal graphs, and belong to a more general class of graphs which we say have extreme local rank 1.
We then provide some additional results that complement these quantitative results and provide stronger bounds in some cases.
Though no deep knowledge of algebraic topology or algebraic geometry is needed to understand the results in this paper, we will be using techniques from algebraic topology to prove these bounds.

\begin{figure}
  \includegraphics[width=0.5\linewidth]{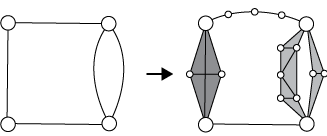}
  \caption{An illustration of the thickened graph construction, illustrating how we may replace the edges of a graph by chordal graphs.}
\end{figure}

\subsection{PSD Completion and Sparse Semidefinite Programming}
One major reason for considering the PSD completion problem is its connection to sparse semidefinite programming.

A semidefinite program is an optimization problem of the form 
\begin{equation*}
\begin{aligned}
    \text{minimize} &&\langle B^0, X\rangle\\
    \text{such that } &&\langle B^{\ell}, X\rangle = b_{\ell} &&\text{ for }\ell \in \{1,\dots, k\}\\
                      &&  X \succeq 0
\end{aligned}
\end{equation*}
Here, the $B^{\ell}$ are $n\times n$ symmetric matrices, $X$ is an $n\times n$ symmetric matrix of variables, and $X \succeq 0$ means that $X$ is PSD, so that all of the eigenvalues of $X$ are nonnegative.

A major difficulty encountered when solving these programs is their high memory usage.
The work of Fukuda in \cite{FUKUDA1970} showed that sparsity properties of a semidefinite program can be exploited to reduce the number of variables.
\cite{zheng2021chordal} has a modern survey.
If $G$ is the graph where $\{i,j\} \in E(G)$ when for some $\ell$, $B^{\ell}_{i,j} \neq 0$, then optimizing a semidefinite program is equivalent to optimizing over over the PSD completable $G$-partial matrices.

It was shown in \cite{blekherman2020sums} that for a certain class of graphs, including the cycle-completable graphs defined in \cite{BARRETT19933}, optimizing over the $G$-locally PSD matrices gives a  $1 + O(\frac{n}{g^3})$ approximation ratio, where $g$ is the number of vertices in the smallest induced cycle with at least 4 vertices in $G$.
This is stated precisely in Theorem 26 in \cite{blekherman2020sums}. We extend these results to a much wider class of graphs here.
\section{Preliminaries and Set Up}
\subsection{Graph Theoretic Preliminaries}
A graph $G$ consists of a set of vertices $V(G)$ and a set of undirected edges $E(G) \subseteq \{\{i,j\} : i, j \in V(G)\}$, and we will assume $G$ has all of its loop edges, i.e. $\{i,i\} \in E(G)$ for all $i \in V(G)$.

For a subset $S \subseteq V(G)$, the subgraph of $G$ induced by $S$ is $G[S]$, where $V(G[S]) = S$, and $E(G[S]) = \{\{i,j\} \in E(G) : i, j \in S\}$.

An induced cycle of $G$ is a subset $S\subseteq V(G)$ so that $G[S]$ is isomorphic to a cycle graph with vertex set $S$.
A graph is said to be \textbf{chordal} if it has no induced cycles with at least 4 vertices.
The \textbf{chordal girth} of $G$ is the number of vertices in the smallest induced cycle of $G$ with more than $4$ vertices, and we will let it be $\infty$ if $G$ is chordal.
A clique of $G$ is a subset $S \subseteq V(G)$ so that $G[S]$ is isomorphic to a complete graph.
We will use the term `triangle' to refer to cliques of size $3$ in $G$, and say that $G$ is triangle free if it contains no triangles.
An induced path in $G$ is a set of vertices $S \subseteq V(G)$ so that the induced subgraph $G[S]$ is isomorphic to a path graph.
A shortest path between vertices $a, b \in V(G)$ is an induced path in $G$ whose endpoints are $a$ and $b$ so that any other path from $a$ and $b$ has at least as many vertices.

For graphs $G$ and $H$, $\phi : V(G) \rightarrow V(H)$ is a graph homomorphism if for $\{i,j\} \in E(G)$, $\{\phi(i), \phi(j) \} \in E(H)$. 

Given an edge $e = \{i,j\} \in E(G)$, we will denote by $G / e$ the contraction of $G$ by the edge $e$, i.e. the graph obtained by identifying the endpoints of the edge $e$ into a single vertex.
Precisely, $G / e$ is a graph defined by
\[
    V(G/e) = (V(G) \setminus e) \sqcup v_e, \text{ and }
\] 
\[E(G/e) = E(G[V(G) \setminus e]) \cup \{\{v,v_e\} : \{v, i\} \in E(G) \text{ or } \{v, j\} \in E(G)\}.\]
There is a natural graph homomorphism $\phi : G \rightarrow G / e$ that sends the vertices in $G$ not in $e$ to themselves, and sends the two vertices in $e$ to $v_e$

A graph $G$ is series parallel if no sequence of edge contractions of $G$ will result in the complete graph on $4$ vertices.
The wheel graph is a graph obtained by taking a cycle, introducing a new vertex, and then joining that vertex to all existing vertices in the cycle.

\subsection{PSD Matrix Completion}
For a graph $G$ on $n$ vertices, a $G$-partial matrix is an $n\times n$ symmetric matrix $X$, where if $i \neq j$ and $\{i,j\} \not \in G$, $X_{ij}$ is set to `unknown'.
The vector space of $G$-partial matrices is denoted $\Sym(G)$.

Given $S \subseteq V(G)$, and $X \in \Sym(G)$, there is a natural restriction of $X$ to the subgraph $G[S]$, which we will denote $X|_S \in \Sym(G[S])$, so that $(X|_S)_{ij} = X_{ij} $ when $\{i,j\} \in E(G[S])$.

A $G$-partial matrix $X$ is said to be PSD-completable if there is a way of choosing the unknown entries to make a complete matrix $\hat{X}$ so that $\hat{X}$ is PSD.
We will denote the convex cone of PSD-completable matrices by $\SOS(G)$.
To differ slightly from the terminology in \cite{blekherman2020sums}, we will say that a $G$-partial matrix is $G$-locally PSD if for each clique $C \subseteq V(G)$, the submatrix $X|_C$ is PSD.
We will denote the convex cone of $G$-locally PSD partial matrices by $\POS(G)$.

For a $G$-partial matrix $X$, we can define the trace by $\tr(X) = \sum_{i=1}^n X_{ii}$.
Let $\tilde{\POS}(G) = \{X \in \POS(G) : \tr(X) = 1\}$.

We will denote the projection of the identity matrix to $\Sym(G)$ by $I_G$.
Given $G$, and a $G$-partial matrix $X$, let
\[
    \epsilon_G(X) = \min \{\epsilon : X + \epsilon(G) I_G \in \SOS(G)\}, \text{ and}
\]
\[
    \epsilon(G) = \min \{\epsilon_G(X) : \forall X \in \tilde{\POS}(G)\}.
\]
We refer to $\epsilon(G)$ as the additive distance for $G$.
$\epsilon(G)$ is in some senses a measurement of how far $X \in \tilde{\POS}(G)$ can be from being PSD completable.
An equivalent definition of $\epsilon_G(X)$ is that $-\epsilon_G(X)$ is the largest possible value of the minimum eigenvalue of $\hat{X}$, where $\hat{X}$ is a completion of $X$.

One reason to consider this $\epsilon(G)$ definition is in the following conical program for some convex cone $K \subseteq \Sym(G)$:
\begin{equation*}
\begin{aligned}
    \text{minimize} &&\langle B^0, X\rangle\\
    \text{such that } &&\langle B^{\ell}, X\rangle = b_{\ell} &&\text{ for }\ell \in \{1,\dots, k\}\\
                      &&  X \in K.
\end{aligned}
\end{equation*}
We say that this program is of \textbf{Goemanns-Williamson type} if $\tr(B^0) = 0$; for all feasible $X$, $\tr(X) \le 1$, and $\frac{1}{n} I_G$ is a feasible point of the program.
\begin{thm}[Theorem 26 in \cite{blekherman2020sums}]
    Let $\alpha$ denote the value of the program when $K = \SOS(G)$, and suppose that it is of Goemanns-Williamson type.
    Let $\alpha'$ denote the value of the program when $K = \POS(G)$, then
    \[ \alpha' \le \alpha \le \frac{1}{1+n\epsilon(G)}\alpha'.\]
\end{thm}

In \cite{blekherman2020sums}, we computed the value of $\epsilon(G)$ for a number of different graphs $G$, including cycles, series parallel graphs, and wheels.
We also provide a few basic operations that interact well with $\epsilon(G)$.
We will be most interested in the following theorems from that paper:
\begin{thm}[Theorem 25 in \cite{blekherman2020sums}]\label{thm:series_parallel}
    Let $G$ be series parallel, and $g$ is the chordal girth of $G$, then
    \[
        \epsilon(G) = \frac{1}{g}\left(\frac{1}{\cos(\frac{\pi}{g})} - 1\right) = O(g^{-3}).
    \]
\end{thm}
\begin{thm}[Theorem 7 in \cite{blekherman2020sums}]
    Let $G$ and $H$ be graphs, so that $V(G) \cap V(H)$ induces a clique in both $G$ and $H$. Let the clique sum of $G$ and $H$ be the union $G \cup H$. Then,
    \[\epsilon(G \cup H) = \max \{\epsilon(G), \epsilon(H)\}.\]
\end{thm}

In \cite{shu2021extreme}, we generalize the definition of $\POS(G)$ to allow for us to enforce the PSD condition on an arbitrary set of cliques of $G$, instead of always enforcing it on all cliques of $G$.
We will not need this level of generality here, but we will introduce a related notion where we only enforce the PSD condition on edges of $G$.
Let $\POS^-(G)$ denote the set of $X \in \Sym(G)$ so that for all $\{i,j\} \in E(G)$, $X|_{\{i,j\}} \succeq 0$.
If $G$ is triangle free, then all cliques of $G$ are edges, so that $\POS^-(G) = \POS(G)$.

Now, let 
\[
    \epsilon^-(G) = \max \{\epsilon_G(X) : X \in \POS^-(G) \text{ and }\tr(X) = 1\}.
\]
Just considering $\epsilon^-(G)$ will not provide useful bounds for interesting graphs $G$, but we will use this definition to bound $\epsilon(G')$ for graphs $G'$ which are related to $G$.

\subsection{Extreme Locally-Rank 1 Graphs and Thickened Graphs}
The extreme rays of $\SOS(G)$ can be seen to be those $G$-partial matrices which are completable to rank 1 PSD matrices.

The extreme rays of $\POS(G)$ are in general harder to understand, but in some cases, they can be seen to have some nice structure.
We will say that $X \in \POS(G)$ is \textbf{locally rank 1} if for all cliques $C \subseteq V(G)$, $X|_C$ is rank 1.
We will say that $G$ has \textbf{extreme local rank 1} if all extreme rays of $\POS(G)$ are locally rank 1.
We will describe these locally rank 1 extreme rays in greater detail in section \ref{sec:extreme_classification}, but it should be clear that this is a very restrictive property.

One natural question is whether there are in fact any interesting graphs of extreme local rank 1.
We now describe the class of \textbf{thickened graphs}, which gives a nice family of examples.

\begin{figure}[H]
    \includegraphics[width=0.25\linewidth]{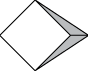}
    \caption{A small example of a thickened graph, which we refer to as a thickened 4-cycle. We will see that the additive distance for this graph is the same as that of the 4-cycle.}
\end{figure}

\subsection{Thickened Graphs}\label{subsec:thick}

\begin{figure}[H]
    \includegraphics[width=0.75\linewidth]{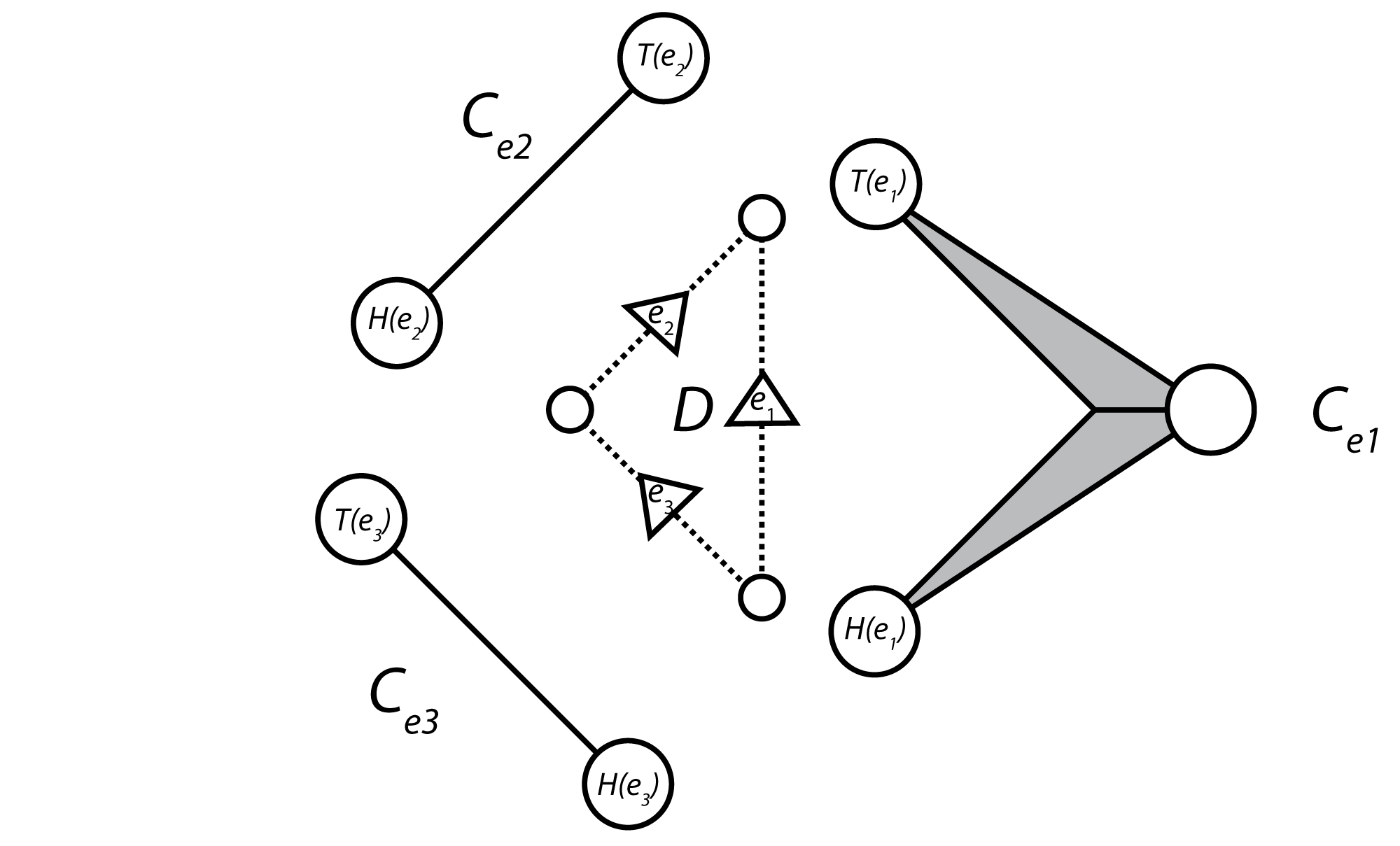}

    \caption{A construction of a thickened 4-cycle. Here, $D$ is a multidigraph with 3 vertices, and 3 edges $e_1, e_2, e_3$. We then introduce chordal graphs $C_{e_1}$, $C_{e_2}$, and $C_{e_3}$ and attach them to form the thickened 4-cycle.\label{fig:thick_graph}}
\end{figure}
A multidigraph (also known as a quiver) $D$ is a directed graph possibly with multiple edges.
Formally, we define a multidigraph in terms of a set of edges, $E(D)$; a set of vertices $V(D)$, and two functions $t : E(D) \rightarrow V(D)$ and $h : E(D) \rightarrow V(D)$ where $t(e)$ is the tail of the edge $e$, and $h(e)$ is the head of the edge $e$, and for each $e \in E(D)$, $t(e) \neq h(e)$.
The undirected graph associated with $D$ is $D'$ where $V(D') = V(D)$ and $E(D') = \{\{i,j\} : \exists e \in E(D), i = t(e) \text{ and } j = h(e)\}$.

Given a multidigraph, $D$, a thickening of $D$ can be thought of informally as a graph obtained by replacing the edges of $D$ by chordal graphs, as seen in figure \ref{fig:thick_graph}.

Formally, for each edge $e \in E(D)$, we take a connected chordal graph $C_e$ and two distinguished vertices $T(e),H(e) \in V(C_e)$.
We then consider the disjoint union of the $C_e$ and $V(D)$, and then identify the vertex $T(e)$ with $t(e)$ and $H(e)$ with $h(e)$ for all $e \in E(D)$, i.e.,
\[
    G = D / \{C_e\} = \left(\sqcup_{e \in E(D)} C_e \sqcup V(D)\right) / \{T(e) \sim t(e), H(e) \sim h(e)\}.
\]
We will say that $G$ is a thickening of $D$ if every triangle in $G$ is contained in $C_e$ for some $e \in E(D)$.
We will say that $G$ is a thickened graph when $D$ is left implicit.

\begin{thm}\label{thm:thick_rank}
    If $G$ is a thickened graph, then $G$ has extreme local rank 1.
\end{thm}

Thickened graphs contain both triangle free and chordal graphs as subclasses, and we will see that it is possible to bound their additive distances.
\section{Results}
\subsection{Additive Distance for Locally Rank 1 Graphs}

Our main result concerns the additive distances for thickened graphs.
In particular, it shows that we can reduce the problem of computing the additive distance for a thickened graph is equivalent to computing the additive distance for an associated triangle-free graph.
Though it is a relatively simple consequence of the more general theorem \ref{thm:homo_bound}, we will state this theorem first because it does not require the definition of simplicial cohomology.
\begin{thm}\label{thm:eps_thick}
    Let $D$ be a multidigraph.
    For each $e \in E(D)$, let $C_e$ be a chordal graph with vertices $T(e), H(e) \in V(C_e)$, and suppose that $D/\{C_e\}$ is a thickening of $D$.
    For each $e\in E(D)$, let $P_e$ be a shortest path from $T(e)$ to $H(e)$ in $C_e$. Then
    \[\epsilon(D / \{C_e\}) = \epsilon(D / \{P_e\}).\]
\end{thm}

These results interact nicely with results from \cite{blekherman2020sums}, and in particular, the previous result, together with theorem \ref{thm:series_parallel} implies the following corollary:
\begin{cor}
    Let $D$ be a multidigraph, so that the undirected version of $D$ is triangle free and series parallel, then for any collection of chordal $C_e$ so that $D / \{C_e\}$ is a thickening of $D$,
    \[\epsilon(D / \{C_e\}) = \epsilon(\mathcal{O}_g).\]
    Here $g$ is the chordal girth of $D / \{C_e\}$, and $\mathcal{O}_g$ is a cycle with $g$ vertices.
\end{cor}
This adds to the a large class of graphs for which $\epsilon(G) = \epsilon(\mathcal{O}_{g})$ for $g$ the chordal girth of $G$ originally described in \cite{blekherman2020sums}.
We will comment that while $\epsilon(G) \ge \epsilon(\mathcal{O}_{g})$ for all graphs $G$, not all graphs meet this bound with equality, and the Peterson graph is a counterexample.
\subsection{Bounds on $\epsilon^-(G)$.}
We round out our results with some bounds on $\epsilon^-(G)$ for some graphs $G$.
Such results complement the results in the previous section.

Let $\omega(G)$ denote the size of the largest clique in $G$, and let $\tw(G)$ denote the treewidth of $G$.
The treewidth of a graph is the smallest possible clique number of a chordal graph with $G$ as a subgraph.

\begin{thm}\label{thm:minus_bound}
    \[1-\frac{2}{\omega(G)}  \le \epsilon^{-}(G) \le 1-\frac{2}{\tw(G) + 1}.\]
\end{thm}
Of course, if $G$ contains a triangle, then this gives us a constant lower bound on $\epsilon^-(G)$, and indicates that $\epsilon^-$ cannot be used to bound this approximation ratio.
However, we can improve these bounds using the \textbf{lengthening} operation.

Let $G$ be a graph, and let $G_{\ell}$ be the result of replacing each edge of $G$ by a path with $\ell$ edges, so that $G_1 = G$.
We formally define this graph in \ref{subsec:lengthening}, but offer a picture for intuition.
\begin{figure}[H]
\includegraphics[width=0.5\linewidth]{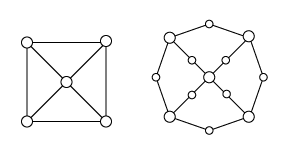}
\caption{An image depicting the lengthening of a wheel graph.}
\end{figure}
Let $\mathcal{O}_{\ell}$ denote the $\ell$-vertex cycle.
\begin{thm}\label{thm:lengthening}
    For $\ell > 1$,
    \[\epsilon(G_{\ell}) \le \frac{\epsilon(\mathcal{O}_{\ell})\epsilon^{-}(G)}{\epsilon(\mathcal{O}_{\ell}) +  2\epsilon^{-}(G)}.\]
\end{thm}

This result can be related to the theorem \ref{thm:eps_thick}, as follows:
\begin{cor}\label{cor:far}
    Let $D$ be a multidigraph, and let $D'$ be the unoriented version of $D$.
    For each $e \in E(D)$, let $C_e$ be a chordal graph with vertices $T(e), H(e) \in V(C_e)$, so that the distance from $T(e)$ to $H(e)$ is at least $\ell$ and $D / \{C_e\}$ is a thickening of $D$.
    Then
    \[\epsilon(D / \{C_e\}) \le  \frac{\epsilon(\mathcal{O}_{\ell})\epsilon^{-}(D')}{\epsilon(\mathcal{O}_{\ell}) +  2\epsilon^{-}(D')}.\]
\end{cor}
Intuitively, this result implies that as long as we use construct a thickened graph using chordal graphs with endpoints which are far apart, we can still get a good bound on the additive distance for $G$, even if $D$ is relatively complicated.
\subsection{Applications}
Typically, semidefinite programs that arise in actual applications do not have a chordal sparsity pattern or even a sparsity pattern associated with a thickened graph.
Instead, one notices that a semidefinite program which is sparse for a graph $G$ will also be sparse for a graph $H$ where $H$ contains $G$, so one finds a larger graph that contains the given sparsity pattern.
Unfortunately, the problem of finding the chordal graph with fewest possible edges containing a given graph $G$ is NP-hard even to approximate (see \cite{CAO2020104514}).

Corollary \ref{cor:far} implies that instead of finding a chordal graph containing $G$, we can instead break $G$ into smaller pieces, and then find a chordal completion for those smaller pieces.
Then, as along as those smaller pieces only intersect in vertices, and if the distances between those intersection points is bounded from below by $\omega(\sqrt{n})$, it is still possible to obtain good upper bounds on the approximation error given by the locally-PSD approximation.
While this is unlikely to yield asymptotic speed ups in the general case, it may still be of interest for practical purposes.

\subsection{Simplicial Cohomology}
In order to state our most general result, we will need a definition of cohomology, where we will refer to \cite{hatcher} for more details.
While we will give an application where cohomology is useful, it is not necessary to understand the result in the case of thickened graphs, which we consider to be the most concrete case.

We will say that a function $f : E(G) \rightarrow \Z / 2 \Z$ is a \textbf{cocycle} if for all $\{i,j,k\} \subseteq V(G)$ so that $\{i,j,k\}$ induces a triangle in $G$, $f(\{i,j\}) + f(\{j,k\}) + f(\{i,k\}) = 0$.
We will say that a function $g : E(G) \rightarrow \Z / 2 \Z$ is a \textbf{coboundary} if there is some function $d : V(G) \rightarrow \Z / 2 \Z$ so that $f(\{i,j\}) = d(i) + d(j)$ for all $\{i,j\} \in E(G)$.
Two cocycles $f_1$ and $f_2$ are said to be equivalent if $f_1 + f_2$ is a coboundary.
The degree 1 simplicial cohomology group of $G$ with $\Z / 2\Z$ coefficients, denoted $H^1(G; \Z / 2\Z)$ is the set of equivalence classes of cocycles of $G$.
\begin{remark}
    Technically, what we are calling the simplicial cohomology group of $G$ is actually the simplicial cohomology group of the \textbf{clique complex} of $G$.
    However, for our purposes, we will abbreviate this to the simplicial cohomology group of $G$.
\end{remark}

The notion of cohomology is extremely pervasive in algebraic topology, and we will see it appear naturally in this context when characterizing the locally rank 1 extreme rays of $\POS(G)$.
One reason that cohomology is so useful is that it interacts nicely with graph homomorphisms.
\subsection{Induced Maps}
If $\phi : G \rightarrow H$ is a graph homomorphism, and $f : E(H) \rightarrow \Z / 2\Z$, then let $\phi^*(f)$ be the function from $E(G)$ to $\Z / 2\Z$ defined by $\phi^*(f)(\{i,j\}) = f(\{\phi(i), \phi(j)\})$.

This is called the \textbf{induced map on cohomology} of $\phi$.
The name is justified because $\phi^*$ sends cocycles to cocycles, and coboundaries to coboundaries, so $\phi^*$ is a group homomorphism from $H^1(H; \Z / 2\Z)$ to $H^1(G; \Z / 2\Z)$.

We will say that $\phi$ is surjective in degree 1 if $\phi^*$ is surjective.

Given $\phi : G \rightarrow H$, and $S\subseteq V(G)$, there is a restricted map from $\phi|_S : G[S] \rightarrow H$ so that for $v \in S$, $\phi|_S(v) = \phi(v)$.
We will say that $\phi$ is \textbf{completely surjective} if for all $S \subseteq V(G)$, $\phi|_S$ is surjective in degree 1.

Completely surjective maps will be useful in constructing our bounds on the additive distances for extreme locally rank 1 graphs.
\subsection{Homomorphism Bound for Graphs of Extreme Local Rank 1}

Our most general, result relates the additive distances of two graphs which are related by a completely surjective graph homomorphism.
\begin{thm}\label{thm:homo_bound}
    If $G$ has extreme local rank 1 and $\phi : G \rightarrow H$ is completely surjective, then $\epsilon(G) \le \epsilon(H)$.
\end{thm}

The remainder of this paper is devoted to proofs of these results.
\section{Classification of Locally Rank 1 Extreme Rays and Cohomology}\label{sec:extreme_classification}
In order to prove our main results, we will need some additional information about the structure of locally rank 1 extreme rays of $\POS(G)$.

The primary thing to notice is that if $X$ is a locally rank 1 partial matrix then because all of the $2\times 2$ minors are rank 1, we have that for all $\{i,j\} \in E(G)$, $X_{ij}^2 = X_{ii}X_{jj}$, so that the values of the $|X_{ij}|$ is determined by the diagonal entries of $X$.
If all of the specified off-diagonal entries of $X$ are nonnegative, and $X$ is locally rank 1, then $X$ is equal to a rank PSD completable matrix.
So a locally rank 1 $X$ is equal to a rank 1 PSD completable partial matrix, \textbf{except the signs of the off-diagonal entries may not be consistent with a PSD-completable partial matrix}.

We will say that $X \in \Sym(G)$ has support $S = \{i \in V(G) : X_{ii} > 0\} = S$.

Given a diagonal matrix $D$, we let $DXD$ be the partial matrix where $(DXD)_{ij} = D_{ii}D_{jj}X_{ij}$.
Notice that if $\hat{X}$ is a completion of $X$, then $D\hat{X}D$ (which is defined using usual matrix multiplication) is a completion of $DXD$.
We will say that partial matrices $X$ and $Y$ are diagonally congruent if $X = DYD$ for some diagonal matrix $D$.
It is not hard to see that if $X \in \POS(G)$ is locally rank 1 with support $S$, then $X$ is diagonally congruent to a locally rank 1 matrix $Y$ where $Y_{ii} = 1$ for $i \in S$; $Y_{ij} = \pm 1$ for $\{i,j\} \in E(G[S])$, and $Y_{ij} = 0$ for $\{i,j\} \not \in E(G[S])$.

Given a locally rank 1 $X \in \Sym(G)$ with support $S$, we will define the sign pattern of $X$ to be the function $f : E(G[S]) \rightarrow \Z / 2\Z$ where $f(\{i,j\}) = \sigma(X_{ij})$.
Here,
\[\sigma(x) = \begin{cases} 0 \text{ if }x \ge 0\\ 1 \text{ otherwise } \end{cases}.\]

The following theorem is shown in \cite{shu2021extreme}.
\begin{thm}
    For any locally rank 1 $X \in \POS(G)$ with support $S$, the sign pattern of $X$ is a cocycle.
    Also, $X, Y \in \POS(G[S])$ are diagonally congruent if and only if the the sign patterns of $X$ and $Y$ differ by a coboundary.
    Finally, for every coycle $f$, there is some locally rank 1 $X \in \POS(G)$ with support $S$ so that the sign pattern of $X$ is $f$.
\end{thm}

\subsection{Graph Homomorphisms, Induced Maps, and Cohomology}
Given a graph homomorphisms $\phi : G \rightarrow H$, define a linear map $\phi^* : \Sym(H) \rightarrow \Sym(G)$ by taking $X \in \POS(H)$ to the partial matrix $Y$ so that $Y_{ij} = X_{\phi(i)\phi(j)}$ for all $\{i,j\} \in E(G)$.
We are abusing this notation because we earlier defined $\phi^*$ to be the induced map on cohomology.
This is justified because if $X \in \POS(H)$ is locally rank 1, then $\phi^*(X) \in \POS(G)$ is also locally rank 1, and moreover if $f$ is the sign pattern of $X$, then $\phi^*(f)$ is the sign pattern of $\phi^*(Y)$ (up to coboundaries).

One key thing about $\phi^*$ is that it also sends $\POS(H)$ to $\POS(G)$ and also it sends $\SOS(G)$ to $\SOS(H)$, i.e
\begin{lemma}\label{lem:pos2pos}
    If $\phi : G \rightarrow H$ then for any $X \in \POS(H)$, $\phi^*(X) \in \POS(G)$.
\end{lemma}
\begin{lemma}\label{lem:sos2sos}
    If $\phi : G \rightarrow H$ then for any $X \in \SOS(H)$, $\phi^*(X) \in \SOS(G)$.
\end{lemma}
These facts are not hard to see from the definitions, and are used for example in \cite{laurent2012gram}.

Now, we introduce the main technical tool in this section, which can be seen from the discussion in this section.
\begin{thm}\label{thm:competely_surjective}
    If $\phi$ is completely surjective, then for every locally rank 1 $X \in \POS(G)$, there is some locally rank 1 $Y \in \POS(H)$ so that $\phi^*(Y)$ is diagonally congruent to $X$.
\end{thm}

\section{Additive Distance for Locally Rank 1 Graphs}
\subsection{Graph Homomorphism Bound}
\begin{thm}[Theorem \ref{thm:homo_bound}]
    Let $G$ and $H$ be extreme locally rank 1 graphs. If $\phi : G \rightarrow H$ is completely surjective, then $\epsilon(G) \le \epsilon(H)$.
\end{thm}
\begin{proof}
    Firstly, note that there exists some extreme point $X$ of $\tilde{\POS}(G)$ so that $\epsilon_G(X) = \epsilon(G)$.
    This follows because $\epsilon_G(X)$ was seen to be a concave function of $X$ in \cite{blekherman2020sums}.

    Every extreme ray of $\POS(G)$ is locally rank 1 by definition, so we know that $X$ is locally rank 1.
    Because $\phi$ is completely surjective, there is some locally rank 1 $Y \in \POS(H)$ so that $\phi^*(Y)$ is diagonally congruent to $X$.
    Assume that $D\phi^*(Y)D = X$ for some diagonal matrix $D$.
    In particular, for any $i$ in the support of $X$, $\phi(i)$ is in the support of $Y$.

    By applying an appropriate diagonal transformation to $Y$, we may assume that
    \[
        Y_{ii} = \sum_{j \in V(G) : \phi(j) = i}X_{ii}.
    \]
    In this case,
    \begin{align*}
        \tr(Y) &= \sum_{i \in V(H)} Y_{ii}\\
        &= \sum_{j \in V(G)}X_{ii}\\
        &= 1.
    \end{align*}
    Therefore, $Y \in \tilde{\POS}(H)$, and from the definition of $\epsilon(H)$, $Y + \epsilon(H) I_H \in \SOS(H)$.

    Therefore, because these induced maps send $\SOS(H)$ to $\SOS(G)$ by Lemma \ref{lem:sos2sos}, $\phi^*(Y + \epsilon(H) I_H) \in \SOS(G)$.
    We can compute then that
    \begin{align*}
        D(\phi^*(Y) + \epsilon(H) \phi^*(I_H))D &= D(\phi^*(Y) + \epsilon(H) \phi^*(I_H))D\\
                                            &= X + \epsilon(H) D\phi^*(I_H)D\\
                                            &\in \SOS(G).
    \end{align*}
    Here, we are using the linearity of $\phi^*$, the fact that $D\phi^*(Y)D = X$ and the fact that diagonal congruence preserves PSD-completability.

    Let $J = D\phi^*(I_H)D$.
    We now apply the following lemma (whose proof we defer until after this proof).
    \begin{lemma}\label{lem:op_norm_J}
        $J$ has a completion whose eigenvalues are all at most 1.
    \end{lemma}
    Using this lemma, we see that $I_G - J$ is PSD completable. Therefore,
    \[
        X + \epsilon(H) I_G = (X + \epsilon(H) J) + \epsilon(H)(I_G - J)
    \]
    is PSD completable, since this is the sum of 2 PSD-completable matrices.
\end{proof}

\begin{proof}(of lemma \ref{lem:op_norm_J})
    Let us consider first the matrix $\phi^*(I_H)$. 
    It can be verified with a direct computation that $\phi^*(I)$ is a partial matrix where for all $\{i,j\} \in E(G)$, 
    \[\phi^*(I_H)_{ij} = \begin{cases} 0 \text{ if }\phi(i) \neq \phi(j) \\ 1 \text{ otherwise}\end{cases}.\]

    Therefore,
    \[J_{ij} = (D\phi^*(I_H)D)_{ij} \begin{cases} 0 \text{ if }\phi(i) \neq \phi(j) \\ D_{ii}D_{jj} \text{ otherwise}\end{cases}.\]

    Consider the completion of $J$ where all of the unknown entries are completed to 0, which we denote $\hat{J}$.
    The key thing thing to notice about $\hat{J}$ is that it can be permuted into a block diagonal matrix, where there is a block $B_v$ for each $v \in V(H)$, where $B_v = \{i \in V(G) : \phi(i) = v\}$.
    That is, $\hat{J}_{ij} = 0$ if $i$ and $j$ do not belong to the same block.

    Therefore, every eigenvalue of $\hat{J}$ is an eigenvalue of $\hat{J}|_{B_v}$ for some $v \in V(H)$.
    Now, we notice that $\hat{J}|_{B_v}$ is a rank 1 matrix, which can easily be seen because for $i,j \in B_v$, $J_{i,j} = D_{ii} D_{jj}$.
    Therefore, there is a unique nonzero eigenvalue of $\hat{J}|_{B_v}$, which is equal to $\tr(\hat{J}|_{B_v})$.
    
    To conclude, we compute
    \begin{align*}
        \tr(\hat{J}|_{B_v}) &=
        \sum_{j \in V(G) : \phi(j) = i} (D\phi^*(I)D)_{jj}\\
        & = \sum_{j \in V(G) : \phi(j) = i} D_{jj}^2\\
        &= \sum_{j \in V(G) : \phi(j) = i} \frac{X_{jj}}{Y_{ii}}\\
        &= \frac{1}{Y_{ii}}\sum_{j \in V(G) : \phi(j) = i} X_{jj}\\
        &= 1.
    \end{align*}
    Therefore, all eigenvalues of $\hat{J}$ are either 1 or 0, which concludes our proof.
\end{proof}

\subsection{Thickened Graphs}
We now turn our attention to thickened graphs specifically out of all graphs of extreme local rank 1.
We will need some homological lemmas to prove our main result in this subsection, Theorem \ref{thm:eps_thick}.

\begin{lemma}\label{lem:vertex_delete}
    Let $G = D / \{C_e\}$ be a thickened graph, and let $S \subseteq V(G)$ be a set of vertices, then the inclusion map $\phi : G[S] \rightarrow G$ is completely surjective.
\end{lemma}
\begin{proof}
    It is not hard to see this directly from the definition of cohomology as we have given it, but there is a simple way to show this using Meyer-Vietoris sequences, whose full definition can be found in \cite{hatcher}.

    It clearly suffices by induction to consider the case when $S = V(G) - v$ for some vertex $v$, and to show that the induced map $\phi^* : H^1(G;\Z / 2\Z) \rightarrow H^1(G[S]; \Z / 2\Z)$ is surjective.

    Let $T$ be the closed neighborhood of $v$ in $G$ (so that $v \in T$), and notice that $G = G[S] \cup G[T]$.
    Part of the Meyer-Vietoris exact sequence states that the following sequence is exact:
    \[
        H^1(G; \Z / 2\Z) \rightarrow 
        H^1(G[S]; \Z / 2\Z) \oplus H^1(G[T]; \Z / 2\Z) \rightarrow H^1(G[S] \cap G[T]; \Z / 2\Z).
    \]
    Our assumptions on $G$ implies that $G[S] \cap G[T]$ is in fact a disjoint union of connected chordal graphs, and therefore, $H^1(G[S] \cap G[T]; \Z / 2\Z) = 0$.

    Therefore, $ H^1(G; \Z / 2\Z)$ surjects onto $H^1(G[S]; \Z / 2\Z) \oplus H^1(G[T]; \Z / 2\Z)$, and in particular, it surjects onto $H^1(G[S]; \Z / 2\Z)$, as desired.
\end{proof}

\begin{lemma}\label{lem:edge_contract}
    Let $G$ be a graph and let $e \in G$ be an edge of $G$ which is not contained in any induced 4-cycles.
    Let $\phi : G \rightarrow G/e$ be the natural graph homomorphism that contracts the edge $e$.
    Then $\phi^* : H^1(G/e; \Z / 2\Z) \rightarrow H^1(G; \Z / 2\Z)$ is surjective.
\end{lemma}
\begin{proof}
    This can actually be seen abstractly by observing that the clique complex of a $G$ deformation retracts to the clique complex of $G/e$.
    We will show this directly for completeness.

     Let $f  \in H^1(G;\Z/2\Z)$ be a cohomology class.
     We need to define some $g \in H^1(G/e; \Z / 2\Z)$ so that $\phi^*(g) = f$.

    To do this, choose a representative cocycle of $f$ so that $f(e) = 0$ (which is easy to producce).
    Let $e = \{i,j\}$ and consider any $v \in G$ so that $\{v,i\}, \{v,j\} \in E(G)$.
    We see that there is a triangle $\{i,j,v\}$ in $G$, so that $f(\{i,j\}) + f(\{i,v\}) + f(\{j,v\}) = 0$, and therefore, since $f(\{i,j\}) = 0$, $f(\{i,v\}) = f(\{j,v\})$.

    Let $g : E(G/e) \rightarrow \Z/2\Z$ so that for all $e \in E(G/e)$, $g(d) = f(t)$ for any $t$ so that $\phi(t) = d$. This does not depend on the choice of $t$ that maps to $d$ by our previous observation.
    It is clear from this definition that if $g$ is a cocycle, then $\phi^*(g) = f$.
    
    It remains to check that $g$ is a cocycle, i.e. that for any triangle $\{a,b,c\} \in G/d$, $g(\{a,b\}) + g(\{b,c\}) + g(\{a,c\}) = 0$.
    We see that if $\{a,b,c\}$ forms a triangle in $G/d$, then there must be some triangle $\{x,y,z\} \in G$ so that $\phi(\{x,y,z\}) = \{a,b,c\}$ from our assumption that $e$ is contained in no induced 4-cycles.
    Therefore, $g(\{v,w\}) + g(\{w,z\}) + g(\{v,z\}) = f(\{v,w\}) + f(\{w,z\}) + f(\{v,z\}) = 0$.

    We conclude that $\phi^*$ is surjective.
    
\end{proof}
\begin{thm}[Theorem \ref{thm:eps_thick}]
    Let $D$ be a multidigraph.
    For each $e \in E(D)$, let $C_e$ be a chordal graph with vertices $T(e), H(e) \in V(C_e)$, and suppose that $D/\{C_e\}$ is a thickening of $D$.
    For each $e\in E(D)$, let $P_e$ be a shortest path from $T(e)$ to $H(e)$ in $C_e$. Then
    \[\epsilon(D / \{C_e\}) = \epsilon(D / \{P_e\}).\]
\end{thm}
\begin{proof}
    Firstly, notice that $P_e$ is an induced subgraph of $C_e$ because it is a shortest path.
    It is clear then that $D / \{P_e\}$ is an induced subgraph of $D / \{C_e\}$, and so it is clear that 
    \[\epsilon(G) \ge \epsilon(D / \{P_e\}).\]

    To get the inequality in the other direction, we will induct on the number of edges of $G$.
    If $|E(G)| = |E(D / \{P_e\})|$, then $G = D / \{P_e\}$, and the equality holds trivially.

    Hence, assume that there is an edge $d \in E(G)$ which is not in $D / \{P_e\}$.
    This implies that for some $e$, $d \in E(C_e)$, so that $d$ is not in $P_e$.
    The next lemma will allow us to apply lemma \ref{lem:edge_contract} in this situation.
    \begin{lemma}\label{lem:special_edge}
        If $G = D / \{C_e\}$ is a thickened graph, and $P_e$ is a shortest path from $T(e)$ to $H(e)$ in $C_e$.
        Suppose that $C_e$ contains an edge not in $P_e$, then there is an edge $d \in C_e$ so that 
        \begin{itemize}
            \item $d$ is not contained in any induced 4-cycles
            \item $d$ is not contained in any shortest path from $T(e)$ to $H(e)$.
        \end{itemize}
    \end{lemma}

    Now, we can apply Lemma \ref{lem:edge_contract} to say that the map $\phi : G \rightarrow G / d$ is surjective in degree 1, and hence by Lemma \ref{lem:vertex_delete}, the map is completely surjective.
    Therefore, by Theorem \ref{thm:homo_bound}, we have that $\epsilon(G) \le \epsilon(G/e)$.
    
    Now, notice that $C_e / d$ is chordal, and therefore, $G/d$ can be viewed as the thickened graph $D / \{C'_e\}$ where $C'_e = C_e$ when $d \not \in C_e$ and $C'_e = C_e / d$ when $d \in C_e$.
    $G/d$ has at least 1 fewer edge than $G$, and because $d$ is contained in no shortest path from $T(e)$ to $H(e)$, we have that $P_e$ is a shortest path from $T(e)$ to $H(e)$.
    Therefore, by induction, we know that $\epsilon(G/d) = \epsilon(D/\{P_e\})$.
    
    Combining these facts, we obtain that $\epsilon(G) = \epsilon(D/\{P_e\})$.
\end{proof}
Before proving the graph theoretic lemma \ref{lem:special_edge}, we will prove a auxilliary lemma:
\begin{lemma}
    Let $G$ be a connected chordal graph, and let $v, w\in G$. Suppose $G$ is not a path graph, then  there exists an edge not contained in any shortest path from $v$ to $w$.
\end{lemma}
\begin{proof}
    Let $P$ be a shortest path from $v$ to $w$, and suppose that the vertices in $P$ are ordered $x_1, \dots, x_k$, where $x_1 = v$ and $x_k = w$, and $\{x_i, x_{i+1}\} \in E(C_e)$ for $i \in [k]$.

    Let $d$ be an edge of $G$ which is a vertex of $P$, and suppose that $d$ is contained in another shortest path, $Q$ from $v$ to $w$.

    It is clear because $P$ and $Q$ have the same start and endpoints that $P \cup Q$ must contain a cycle, say $\mathcal{O} = \{x_a, \dots, x_b\} \cup \{y_c, \dots, y_d\}$, where $y_c = x_a$ and $y_d = x_b$.
    (It is not hard to see that we can assume that $a<b$ and $c<d$ using some basic observations about shortest paths.)
    Notice that the number of vertices of $P$ in $\mathcal{O}$ must be the same as the number of vertices of $Q$ in $\mathcal{O}$.
    This follows because $P$ and $Q$ are both shortest paths, so any subpath is a shortest path, and in particular, both $x_a\dots, x_b$ and $y_c, \dots, y_d$ are shortest paths of the same length.
    In particular, this cycle has an even number of vertices.

    Based on this observation, we can perform a simple surgery to ensure that in fact $P \cup Q$ contains a unique cycle.
    In $Q$, replace $y_1, \dots, y_{c-1}$ by $x_1, \dots, x_{a-1}$ and replace $y_{d+1}, \dots, y_k$ by $x_{b+1},\dots, x_{c+1}$.

    In summation, we have found two shortest paths $P = \{x_1, \dots, x_k\}$ and $Q = \{y_1, \dots, y_k\}$ where $y_i \neq x_i$ iff $i \in [a+1,b-1]$.
    Inside $P\cup Q$, there is a cycle $\mathcal{O} = \{x_a, \dots, x_b\} \cup \{y_a, \dots, y_b\}$.

    Now, this unique cycle cannot have fewer than 4 vertices because it has an even number of vertices, so because $G$ is chordal, there must be some edge $\{x_e, y_f\} \in E(G)$, where $e,f \in [a+1, b-1]$.
    We wil call this edge a chord of $\mathcal{O}$
    Suppose that this chord is the edge $\{x_e, y_e\}$ for some $e$.
    We claim that such an edge is not in any shortest paths from $v$ to $w$ in $C_e$.
    
    Suppose otherwise, and that there was a path $z_1, \dots, z_k$ where $z_c = x_e$ and $z_{c+1} = y_e$.
    Clearly, either $c < e$ or $c+1 > e$.
    If the former is the case then $z_1, \dots, z_c$ is a shorter path from $v$ to $x_e$ than $x_1, \dots, x_e$, and if the latter is the case, then $z_{c+1}, \dots, z_k$ is a shorter path from $y_e$ to $w$ than $y_e, \dots, y_k$.
    Either of these would be a contradiction.

    Therefore, we suppose that the chord is of the form $\{x_e, y_f\}$ for $e \neq f$.
    In this case, we can either replace $Q$ by the shortest path $x_1, \dots, x_e, y_f, \dots, y_k$, or $P$ by the shortest path $y_1, \dots, y_f, x_e, \dots, x_k$.
    The difference in the number of vertices between these two paths decreases in either case, so by doing this operation repeatedly, we may assume that $P$ and $Q$ differ in exactly one vertex, say $x_{\ell} \neq y_{\ell}$.
    In that case, $\{x_{\ell-1}, x_{\ell}, y_{\ell}, y_{\ell+1}\}$ contains a 4 cycle, which must have a chord, and this chord must be $\{x_{\ell}, y_{\ell}\}$, giving the desired edge.

\end{proof}
\begin{proof} (of lemma \ref{lem:special_edge})
    The basic idea in this proof is that any induced 4-cycle of $G$ cannot be contained in $C_e$ because $C_e$ is chordal, and hence, if it contains any vertex of $C_e$, it must contain $T(e)$ and $H(e)$ as well.

    We will need to divide this into 4 cases, depending on the length of the shortest path from $T(e)$ to $H(e)$.

    \begin{enumerate}
        \item The distance from $T(e)$ to $H(e)$ is 1 in $C_e$.
        \item The distance from $T(e)$ to $H(e)$ is 2 in $C_e$.
        \item The distance from $T(e)$ to $H(e)$ is 3 in $C_e$.
        \item The distance from $T(e)$ to $H(e)$ greater than 3 in $C_e$.
    \end{enumerate}

    Case 1 is trivial; any edge in $C_e$ which is not $\{T(e), H(e)\}$ will satisfy the conditions, as it is clearly not in a shortest path from $T(e)$ to $H(e)$, and also any induced 4-cycle in $G$ would need to be contained in $C_e$, which is chordal.

    In case 2, we have some vertex $v \in C_e$ so that $\{T(e), v\}$ and $\{v, H(e)\}$ are both in $C_e$.
    We notice that $\{T(e), H(e)\}$ is not contained in $G$ in that case, because we assumed in our definition of thickened graphs that all triangles in $G$ are contained in $C_e$ for some $e$.
    Therefore, the only way for a vertex $v \in V(C_e)$ to be contained in an induced 4-cycle is if $\{v, T(e)\}$ and $\{v, H(e)\}$ are both edges of $G$.
    Hence, it suffices to find an edge of $C_e$ not contained in any paths of length 2 from $T(e)$ to $H(e)$, which follows from the previous lemma

    In case 3, we again claim that as long as we find a vertex which is not on a path of length 3 in $C_e$, then we are done.
    To see this, notice that a 4-cycle containing a vertex of $C_e$ would need to contain both $T(e)$ and $H(e)$, and because there are no paths of length 2 from $T(e)$ to $H(e)$, there must be a path of length 1 from $T(e)$ to $H(e)$.
    Because this edge is not an edge of $C_e$, an edge of $C_e$ contained in a 4-cycle must be contained in a path of length 3 from $T(e)$ to $H(e)$ in $C_e$.
    Thus, the result follows from the previous lemma.

    In our final case, once again, it suffices to appeal to the previous lemma, simply because any 4 cycle containing $T(e)$ and $H(e)$ can include at most 2 other vertices, and there are no paths of with 2 vertices from $T(e)$ to $H(e)$ in $C_e$.
\end{proof}
\section{Bounds on $\epsilon^-(G)$}
\begin{thm}[Theorem \ref{thm:minus_bound}]
    \[1-\frac{2}{\omega(G)}  \le \epsilon^{-}(G) \le 1-\frac{2}{\tw(G) + 1}.\]
\end{thm}
\begin{proof}
    The key observation here is shown in \cite{blekherman2020hyperbolic}, that if $C$ is a complete graph on $n$ vertiecs, then $\epsilon^-(C) = 1-\frac{2}{n}$ (though it is stated in a slightly different language).

    Let $C \subseteq G$ be a clique of size $n = \omega(G)$.
    It is clear that $\epsilon^{-}(G) \le \epsilon^-(C) = 1-\frac{2}{n}$, as desired.

    On the other hand, let $H$ be a chordal graph containing $H$ whose clique number is at most $\tw(G)+1$.
    Let $X \in \POS^-(G)$, and let $X' \in \POS^-(H)$ be any $H$-partial matrix so that $X'_{ij} = X_{ij}$ for $\{i,j\} \in E(G)$, and $X_{ij} = 0$ otherwise.

    If $C \subseteq H$ is any clique of $H$ of size $n$, then $X'|_C \in \POS^-(C)$, so that
    \[
        X'|_C + (1-\frac{2}{n})I_C \succeq 0.
    \]

    In particular, if we let $n = \tw(G+1) = \omega(H)$, then we obtain that for all cliques $C$ of $H$, $(X' +  (1-\frac{2}{n})I_H)|_C \succeq 0$.

    Therefore, $X' + (1-\frac{2}{\tw(G)+1})I_H$ is an $H$-partial matrix so that for every clique, $C \subseteq H$, the corresponding submatrix is PSD.
    By the classic theorem in \cite{GRONE1984109}, this impies that $X' + (1-\frac{2}{\tw(G)+1})I_H$ is PSD completable, and in particular, $X + (1-\frac{2}{\tw(G)+1})I_G$ is PSD completable.
\end{proof}

\subsection{Lengthening Graphs}\label{subsec:lengthening}
Let $G$ be a graph, and let $G_{\ell}$ be the result of replacing each edge of $G$ by a path with $\ell$ edges, so that $G_1 = G$.
Formally, we can define $G_{\ell}$ as follows: for each edge $e = \{u,v\} \in E(G)$, we introduce $\ell+1$ new vertices $x_{e,1}, x_{e,2},\dots, x_{e,\ell+1}$, set $x_{e,1} = u$ and $x_{e,2} = v$, and include the edges $\{x_{e,i}, x_{e,{i+1}}\} \in G_{\ell}$ for $i \in [\ell]$.

Notice though that for $\ell > 1$, $G_{\ell}$ is triangle free; in particular, $\POS^-(G) = \POS(G)$.
For each $e \in G$, let $P_e$ be the path of length $\ell$ that replaces the edge $e \in G$.
\begin{thm}[Theorem \ref{thm:lengthening}]
    For $\ell > 1$,
    \[\epsilon(G_{\ell}) \le \frac{\epsilon(\mathcal{O}_{\ell})\epsilon^{-}(G)}{\epsilon(\mathcal{O}_{\ell}) +  \epsilon^{-}(G)}.\]
\end{thm}
\begin{proof}
    The idea of the proof is this: consider the graph $H = G_{\ell} \cup G$.
    We see that $H$ is the union of a number of cycles of the form $\mathcal{O}_e = P_e \cup e$.
    We can think of this as being $H$ as being the result of performing clique sums of $G$ with a cycle of length $\ell+1$ for each edge.

    Then, given $X \in \POS^-(G_{\ell})$, we first complete $X$ to $X' \in \POS^-(H)$, in such a way that for each edge $e \in G$, we can have a lower bound on the minimum eigenvalue of the submatrix $X|_{e}$.
    It is clear that this should make it easier in some sense to complete $X'|_{G}$, but doing this may make completing $X'|_{\mathcal{O}_e}$ harder.

    We then must balance the cost of completing the various cycles of $H$ with the cost of completing $X'|_{G}$.

    So, fix some $X \in \POS^-(G_{\ell})$ with $\tr(X) = 1$, and assume that $X$ is extreme, so that $X|_{\{i,j\}}$ is rank 1 for all $\{i,j\} \in E(G_{\ell})$.

    Now, for each $e \in E(G)$, there is a unique PSD completion of $X|_{P_e}$, which is rank 1.
    Let $x_e$ be the entry of this completion associated with the edge $e \in E(G)$.

    Consider the partial matrix $X' \in \POS(H)$ where $X'_{ij} = X_{ij}$ for $\{i,j\} \in E(G_{\ell})$, and where $X_{ij}' = (1-\delta)x_{\{i,j\}}$ for $\{i,j\} \in E(G)$, where $\delta$ is to be chosen later.

    We want to make two claims:
    
    \begin{lemma}\label{lem:complete_cycles}
    For each $e \in E(G)$, let $\mathcal{O}_e$ be the cycle obtained by taking $e \cup P_e$, then 
    \[
        \epsilon_{\mathcal{O}_e}(X'|_{\mathcal{O}_e}) \le \delta \frac{\epsilon(\mathcal{O}_{\ell})}{2}.
    \]
    \end{lemma}
    \begin{lemma}\label{lem:complete_graph}
        \[\epsilon_{G}(X'|_{V(G)}) \le (1-\delta)\epsilon^{-}(G).\]
    \end{lemma}

    If we let $\epsilon = \max \{\delta \epsilon(\mathcal{O}_e), (1-\delta)\epsilon^-(G)\}$, then consider $X' + \epsilon I_H$.
    Our construction implies that $(X' + \epsilon I_H)|_{\mathcal{O}_e}$ is PSD completable for all $e$, and $(X' + \epsilon I_H)|_{G}$ is PSD completable.
    Since $H$ is a clique sum of $G$ and the $\mathcal{O}_e$, by Theorem 7 in \cite{blekherman2020sums}, we have that $X' + \epsilon I_{H}$ is PSD completable.
    In particular, this implies that $X + \epsilon I_{G_{\ell}}$ is PSD completable.

    Optimizing for the value of $\epsilon$ with respect to $\delta$, we obtain that it suffices if
    \[
        \epsilon = \frac{\epsilon(\mathcal{O}_{\ell})\epsilon^-(G)}{\epsilon(\mathcal{O}_{\ell})+2\epsilon^-(G)}.
    \]
    This is what we wanted.
\end{proof}
\begin{proof}(of lemma \ref{lem:complete_cycles})
    Firstly, we will reindex $X'|_{\mathcal{O}_e}$, so that the vertices of $\mathcal{O}_{e}$ are $1,\dots, \ell+1$.

    If we consider $X'|_{P_e}$, we see that this is a $P_e$-partial matrix where all of the specified $2\times 2$ minors are rank 1. Because $P_e$ is chordal, $X'|_{P_e}$ is PSD completable, and any PSD completion must have rank 1.
    Call this PSD completion $\hat{Q}$, and let $Q$ be the $\mathcal{O}_e$-partial matrix obtained by projecting $Q$ onto the edges of $\mathcal{O}_e$.

    Now, let $A = X'|_{\mathcal{O}_e}$, and notice that we defined $A_{1, \ell+1}$ to be $(1-\delta)Q_{1, \ell+1}$.
    Let $Z$ be the $\mathcal{O}_e$-partial matrix where $Z_{ij} = A_{ij}$ unless $\{i,j\} = \{1,\ell+1\}$, and $Z_{1,\ell+1} = -Q_{1,\ell+1}$.
    Also, notice that $Z \in \POS(\mathcal{O}_e)$, which can be easily checked by noticing that for $2\times 2$ PSD matrices, we can negate the off-diagonal entry and it will still be PSD.
    Moreover, $\tr(Z) = \tr(A) \le \tr(X') = 1$.
    Therefore, $\epsilon_{\mathcal{O}_e}(Z) \le \epsilon(\mathcal{O}_e)$.
    
    We see that $A = (1-\frac{\delta}{2})Q + \frac{\delta}{2} Z$, and therefore, because $\epsilon_{\mathcal{O}_e}(X)$ is concave,
    \[
        \epsilon_{\mathcal{O}_e}(A) \le (1-\frac{\delta}{2})\epsilon_{\mathcal{O}_e}(Q) + \frac{\delta}{2} \epsilon_{\mathcal{O}_e}(Z) \le \frac{\delta}{2} \epsilon(\mathcal{O}_e).
    \]
\end{proof}
\begin{proof}(of lemma \ref{lem:complete_graph})
    Let $A = X'|_G$.
    $A$ is defined so that $A_{ii} = X_{ii}$, and for $\{i,j\} \in E(G)$, $A_{ij} = (1-\delta) x_{ij}$, where $x_{ij}$ is such that
    \[\begin{pmatrix} A_{ii} & x_{ij} \\ x_{ij} & A_{jj} \end{pmatrix}.\]

    Consider the $G$-partial matrix $Q$ where $Q_{ii} = A_{ii}$ and $Q_{ij} = x_{ij}$ for $\{i,j\} \in E(G)$.
    $\tr(Q) = \tr(X|_{G}) \le \tr(X) = 1$, so that $\epsilon_G(Q) \le \epsilon^-(G)$.

    On the other hand, let $D$ be the $G$-partial matrix where $D_{ii} = A_{ii}$ for $i \in V(G)$, and $D_{ij} = 0$ otherwise. 
    This is clearly completable to a positive semidefinite diagonal matrix.
    We see that $A = \delta D + (1-\delta)Q$, so that by concavity we obtain
    \[
        \epsilon_G(A) \le \delta \epsilon_G(D) + (1-\delta)\epsilon_G(Q) \le (1-\delta)\epsilon^-(G).
    \]
\end{proof}

\bibliographystyle{plain}
\bibliography{main}
\end{document}